\documentclass[12pt]{amsart}
\usepackage[margin=1.25in]{geometry}
\usepackage{latexsym}
\usepackage{amsmath}
\usepackage{amssymb}
\usepackage{amsthm}
\usepackage{stmaryrd}
\usepackage{amscd}
\usepackage{enumerate} 
\usepackage{amssymb} 
\usepackage{mathrsfs}
\usepackage[all]{xy}
\usepackage{color}
\usepackage{graphicx}
\usepackage{mathtools}
\usepackage{comment}
\usepackage{multirow}
\usepackage{hyperref}
\hypersetup{colorlinks=true}

\makeatletter
  
  \@addtoreset{equation}{thm}
  \makeatother

\title{On the Kawamata-Viehweg vanishing theorem for log Calabi-Yau surfaces in large characteristic}
\author{Tatsuro Kawakami}
\email{tatsurokawakami0@gmail.com}
\address{Department of Mathematics, Graduate School of Science, Kyoto University, Kyoto 606-8502, Japan}

\def\phi{\varphi}
\def\epsilon{\varepsilon}

\def\mapsto{\longmapsto}

\def\log{\operatorname{log}}

\def\Spec{\operatorname{Spec}}

\def\Supp{\operatorname{Supp}}

\def\Exc{\operatorname{Exc}}

\def\max{\operatorname{max}}

\newcommand{\Q}{\mathbb{Q}} 
 
\newcommand{\R}{\mathbb{R}} 
\newcommand{\Z}{\mathbb{Z}}
\newcommand{\PP}{\mathbb{P}}

\newcommand{\sO}{\mathcal{O}}

\theoremstyle{plain}
\newtheorem{thm}{Theorem}[section] 

\newtheorem{prop}[thm]{prop}

\newtheorem{lem}[thm]{Lemma}
\theoremstyle{definition} 
\newtheorem{defn}[thm]{Definition}

\theoremstyle{remark}
\newtheorem{rem}[thm]{Remark}

\newtheorem{defn and notation}[thm]{Definition and Notation}
\newtheorem*{notation}{Notation} 
\newtheorem*{cl}{Claim}
\newtheorem*{clproof}{Proof of Claim}

\theoremstyle{plain}
\newtheorem{theo}{Theorem}

\begin{document}
\tolerance = 9999

\maketitle
\markboth{Tatsuro Kawakami}{Kawamata-Viehweg vanishing on log Calabi-Yau surfaces}

\begin{abstract}
We prove that the Kawamata-Viehweg vanishing theorem holds for a log Calabi-Yau surface $(X, B)$ over an algebraically closed field of large characteristic when $B$ has standard coefficients. 
\end{abstract}

%\setcounter{tocdepth}{1}
%\tableofcontents

\section{Introduction}

The Kawamata-Viehweg vanishing theorem is an essential tool in birational geometry, but it does not hold in general in positive characteristic. For example, Cascini and Tanaka \cite{CT} clarified that there exists a smooth rational surface over an algebraically closed field of any characteristic that violates the Kawamata-Viehweg vanishing theorem. 
Log Fano varieties and log Calabi-Yau varieties naturally appear in the minimal model program and are significant objects in birational geometry. Therefore, it is important to investigate the Kawamata-Viehweg vanishing theorem on log Fano varieties and log Calabi-Yau varieties in positive characteristic.

Recently, significant progress has been made in the Kawamata-Viehweg vanishing theorem on a log del Pezzo surface, a two-dimensional log Fano variety. Cascini, Tanaka, and Witaszek \cite{CTW} proved the existence of a positive integer $p_0$ such that the Kawamata-Viehweg vanishing theorem holds on a log del Pezzo surface over an algebraically closed field of characteristic $p>p_0$. Additionally, Arvidsson, Bernasconi, and Lacini \cite{ABL} showed that $p_0=5$ is the optimal bound. The Kawamata-Viehweg vanishing theorem on a log del Pezzo surface has many important applications, including the proof of Cohen-Macaulayness of three-dimensional klt singularities in characteristic $p>5$. We refer to \cite{ABL}, \cite{Bernasconi-Kollar}, \cite{Hacon-Witaszek}, and the references therein for the details.

In this paper, our objective is to establish the existence of an integer $p_0$ such that the Kawamata-Viehweg vanishing theorem holds on a log Calabi-Yau surface $(X,B)$ over an algebraically closed field of characteristic $p>p_0$. Specifically, we determine such a value for $p_0$ when $B$ has standard coefficients. By standard coefficients, we mean that the coefficients of $B$ belong to $\{\frac{m-1}{m}|m\in\Z_{>0}\ \cup \infty \}$, where we define $\frac{\infty-1}{\infty}$ to be equal to $1$.

\begin{theo}\label{Intro, KVV}
There exists a positive integer $p_0\in\Z_{>0}$ with the following property:
Let $(X, B)$ be a log Calabi-Yau surface over an algebraically closed field of characteristic $p>p_0$ such that $B$ has standard coefficients.
Let $\Delta$ be a $\Q$-divisor such that $\Supp(\Delta)\subset \Supp(B)$ and $\lfloor \Delta\rfloor=0$. 
Let $D$ be a $\Z$-divisor such that $D-(K_X+\Delta)$ is nef and big. 
Then \[
H^i(X, \sO_X(D))=0
\] 
for all $i>0$.
\end{theo}
\begin{rem}
Arvidsson, Bernasconi, and Lacini \cite{ABL} made use of Lacini's classification \cite{lac} of klt del Pezzo surfaces in positive characteristic with Picard rank one in order to determine the optimal bound $p_0=5$. This allowed them to establish that every log del Pezzo surface over an algebraically closed field of characteristic $p>5$ satisfies the Kawamata-Viehweg vanishing theorem.

In a similar vein, if we can obtain a comparable classification of klt Calabi-Yau surfaces in positive characteristic, it may provide an explicit value of $p_0$ in Theorem \ref{Intro, KVV}.
However, to the best of the author's knowledge, such a classification is currently unknown, even in characteristic zero.
\end{rem}

Recalling that the Kawamata-Viehweg vanishing theorem on log del Pezzo surfaces played an important role in investigating three-dimensional klt singularities (\cite[Section 1]{ABL}), Theorem \ref{Intro, KVV} is expected to have applications in the study of three-dimensional lc singularities.

In order to prove Theorem \ref{Intro, KVV}, we prove the \textit{log liftability} of log Calabi-Yau surfaces.

\begin{defn}[Log liftability]
Let $k$ be an algebraically closed field of positive characteristic.
Let $(X, B)$ be a pair of a normal projective surface $X$ over $k$ and a $\Q$-divisor $B$ on $X$.
We say that $(X, B)$ is \textit{log liftable} if there exists a log resolution $f\colon Y \to X$ of $(X, B)$ such that $(Y,f_{*}^{-1}\Supp(B)+\Exc(f))$ lifts to the ring $W(k)$ of Witt vectors.
For the definition of liftability of a log smooth pair, we refer to \cite[Definition 2.6]{Kaw3}.
\end{defn}

\begin{theo}\label{Intro:Thm:log liftability of log CY}
Let $I\subset [\frac{1}{2},1]\cap \Q$ be a DCC set of rational numbers.
There exists a positive integer $p(I)\in\Z_{>0}$ depending only on $I$ with the following property:
For every log Calabi-Yau surface pair $(X, B)$ over an algebraically closed field of characteristic $p>p(I)$ such that the coefficients of $B$ belong to $I$, the pair $(X, B)$ is log liftable.
\end{theo}

\begin{rem}\,
\begin{enumerate}
\item 
Let $I$ be a finite set.
Cascini, Tanaka, and Witaszek \cite[Theorem 1.1]{CTW} proved the existence of a positive integer $p(I)\in\mathbb{Z}_{>0}$ depending only on $I$ with the following property: For every log del Pezzo surface $(X,B)$ over an algebraically closed field of characteristic $p>p(I)$ such that the coefficients of $B$ belong to $I$, the pair is either log liftable or globally $F$-regular. However, it remains unknown whether a globally $F$-regular surface pair $(X,B)$ is log liftable.
(Note that this is true when $B=0$ by \cite[Theorem 1.3]{BBKW}.)
As a result, the question of the log liftability of log del Pezzo surface pairs remains open (also see \cite[Conjecture 6.20]{KTTWYY1} for further discussion).
\item For the case without boundary, i.e., the log liftability of normal projective surfaces, we refer to \cite[Theorem 1.3]{Kaw3}.
\end{enumerate}
\end{rem}

\subsection{Sketch of proof of Theorem \ref{Intro:Thm:log liftability of log CY}}
The author \cite[Section 3]{Kaw3} provided a proof of Theorem \ref{Intro:Thm:log liftability of log CY} when $X$ is klt and $B=0$. The proof can be outlined as follows:

\textbf{Step 0:} Consider a klt log Calabi-Yau surface pair $(X,0)$ (see Definition \ref{def:log CY}). For simplicity, we assume that $X$ is not canonical.

\textbf{Step 1:} Begin by choosing an extraction $f\colon Y\to X$ of a divisor $E$ over $X$ with the maximum coefficient $e(E,X,0)$. This step reduces the problem of the log liftability of $X$ to that of the pair $(Y,E)$ (see Definition \ref{def:epsilon-klt} for the definition of $e(E,X,0)$).

\textbf{Step 2:} Show the existence of a positive real number $\epsilon\in\mathbb{R}_{>0}$, independent of $(Y,e(E,X,0)E)$, such that the pair $(Y,e(E,X,0)E)$ is $\epsilon$-klt (cf.~\cite[Lemma 3.9]{Kaw3}).

\textbf{Step 3:} Run a $K_Y$-MMP $\phi\colon Y\to Y'$ to obtain a $K_{Y'}$-Mori fiber space $Y'\to Z$ and reduce the log liftability of the pair $(Y,E)$ to that of $(Y',E'\coloneqq \phi_{*}E)$.

\textbf{Step 4:} Utilize the techniques from \cite{CTW} to prove the boundedness of pairs $(Y',E')$ as above. The log liftability of $(Y',E')$ is then deduced from this boundedness (\cite[Lemma 3.14]{Kaw3}).\\

If $(X,B)$ is klt, the above proof remains valid even when $B\neq 0$. However, when $(X,B)$ is not klt, there exist log Calabi-Yau surface pairs $(Y',E')$ that admit $K_{Y'}$-Mori fiber structures, but do not form a bounded family (see \cite[Example 1.11]{GR}). As a result, Step 4 fails in this case. To address this issue, the following strategy is adopted.\\

\textbf{Step 0:} Consider a log Calabi-Yau surface pair $(X,B)$ such that $(X,B)$ is not klt and $B$ has standard coefficients.

\textbf{Step 1:} Begin by taking a dlt blow-up $f\colon (Y,B_Y\coloneqq f_{*}^{-1}B+\text{Exc}(f))\to (X,B)$, which reduces the problem of log liftability from $(X,B)$ to $(Y,B_Y)$.

\textbf{Step 2:} Show the existence of a positive real number $\epsilon\in\mathbb{R}_{>0}$, independent of $(Y,B_Y)$, such that the pair $(Y, B_Y^{<1})$ is $\epsilon$-klt (Proposition \ref{prop:epsilon-dlt}). Here, $B_Y^{<1}$ denotes the sum of the irreducible components of $B_Y$ with coefficients less than one.

\textbf{Step 3:} Run a $(K_Y+B_Y^{<1})$-MMP to obtain a birational contraction $\phi\colon Y\to Y'$ and a $(K_{Y'}+B_{Y'}^{<1})$-Mori fiber space $(Y', B_{Y'}\coloneqq \phi_{*}B_Y)\to Z$. This step reduces the log liftability of $(Y,B_Y)$ to that of $(Y',B_{Y'})$.

\textbf{Step 4:} If $\dim\,Z=0$, prove the boundedness of pairs $(Y',B_{Y'})$ satisfying the conditions and deduce the log liftability of $(Y',B_{Y'})$ from this boundedness.

\textbf{Step 4':} If $\dim\,Z=1$, the pair $(Y',B_{Y'})$ does not form a bounded family in general (see \cite[Example 1.11]{GR}). However, we can show that $H^2(Y', T_{Y'}(-\log\,\Supp(B_{Y'})))=0$, which contains the obstruction for the log lifting of $(Y',B_{Y'})$. To prove this vanishing result, we require the fact that $B_{Y'}$ has standard coefficients, and there exists an irreducible component $C$ of $B_{Y'}$ with $\mathrm{coeff}_C B_{Y'}=1$ and $C$ dominates $Z$. To find such a component $C$, we run a $(K_Y+B_Y^{<1})$-MMP instead of a $K_X$-MMP in Step 3.\\

This revised strategy aims to address the challenges encountered in Step 4 of the previous approach by incorporating additional considerations.

\begin{notation}
A \textit{variety} is defined as an integral separated scheme of finite type over a field.
% Unless otherwise stated, we assume that a variety is defined over an algebraically closed field $k$ of characteristic $p>0$.
A \textit{curve} (resp.~a \textit{surface}) is a variety of dimension one (resp.~two).
For a proper birational morphism $\phi\colon X\to X'$ between normal surfaces and a $\Q$-divisor $D'$ on $X'$, we denote the Mumford pullback by $f^{*}D'$.
A \textit{pair} $(X,B)$ consists of a normal variety $X$ and an effective $\Q$-divisor $B$ such that the coefficients of $B$ belong to $[0,1]$.
We say that a pair $(X,B)$ is \textit{projective} if $X$ is projective.
We say that a pair $(X,B)$ is \textit{log smooth} if $X$ is smooth and $B$ has simple normal crossing support.
For definitions of the singularities appearing in the minimal model program, we refer to \cite[Section 2.3]{Kollar-Mori}.
Throughout this paper, we will use the following notation:
\begin{itemize}
\item $\Exc(f)$: the reduced exceptional divisor of a proper birational morphism $f$.
\item $\lfloor B \rfloor$ (resp.$\lceil B \rceil$): the \textit{round-down} (resp.\textit{round-up}) of a $\Q$-divisor $B$.
\item $\{ B \} \coloneqq B - \lfloor B \rfloor$: the \textit{fractional part} of a $\Q$-divisor $B$.
\item $B^{=1} = \sum_{b_i=1} b_i B_i$: the sum of the irreducible components of a $\Q$-divisor $B = \sum_i b_i B_i$ with coefficient one.
\item $B^{<1} = \sum_{b_i<1} b_i B_i$: the sum of the irreducible components of a $\Q$-divisor $B = \sum_i b_i B_i$ with coefficient less than one.
\item $\mathcal{F}^{*}$: the dual of a coherent sheaf $\mathcal{F}$.
\item $\Omega_X^{[i]}(\log B)$: the $i$-th logarithmic reflexive differential form $j_{*}\Omega_U^i(\log B)$, where $j\colon U \hookrightarrow X$ is the inclusion of the log smooth locus of a pair $(X, B)$ consisting of a normal variety $X$ and a reduced divisor $B$ on $X$.
\item $T_X(-\log B) \coloneqq (\Omega_X^{[1]}(\log B))^{*}$: the \textit{logarithmic tangent sheaf} for a pair $(X, B)$ consisting of a normal variety $X$ and a reduced divisor $B$ on $X$.
\item $W(k)$: the ring of Witt vectors.
\end{itemize}
\end{notation}

\section{Preliminaries}

\subsection{Log Calabi-Yau surfaces}

In this subsection, we will review fundamental properties of log Calabi-Yau surfaces.

\begin{defn}\label{def:epsilon-klt}
Let $(X,B)$ be a surface pair over an algebraically closed field.
Let $E$ be a prime divisor over $X$, and let $f\colon Y\to X$ be a proper birational morphism from a normal surface $Y$ such that $E$ is a divisor on $Y$.
The coefficient of $E$ in $f^{*}(K_X+B)-K_Y$ is denoted by $e(E,X,B)$. Note that $e(E,X,B)$ does not depend on the choice of $f$.

We fix a real number $\epsilon\in\R_{>0}$.
We say that a pair $(X, B)$ is \textit{$\epsilon$-klt} if $K_X+B$ is $\Q$-Cartier and $e(X,B,E)$ is less than $1-\epsilon$ for every divisor $E$ over $X$.
\end{defn}

\begin{defn}\label{def:log CY}
Let $(X, B)$ be a projective surface pair over an algebraically closed field.
We say that $(X, B)$ is \textit{log del Pezzo} if $(X, B)$ is klt and $-(K_X+B)$ is ample.
We say that $(X, B)$ is \textit{log Calabi-Yau} if $(X, B)$ is lc and $K_X+B\equiv 0$.
We say that $(X, B)$ is \textit{dlt (resp.~$\epsilon$-klt) log Calabi-Yau} if $(X,B)$ is dlt (resp.~$\epsilon$-klt) and $K_X+B\equiv 0$.
\end{defn}

\begin{lem}\label{lem:log CY}
    Let $(X,B)$ be a projective surface pair over an algebraically closed field.
    Let $\phi\colon X\to X'$ be a birational morphism to a normal projective surface $X'$, and $B'\coloneqq \phi_{*}B$.
    Let $\epsilon\in\R_{>0}$ be a real number. 
    Suppose that $K_{X'}+B'$ is $\Q$-Cartier.
    If $(X, B)$ is log Calabi-Yau (resp.~$\epsilon$-klt log Calabi-Yau), then so is $(X',B')$.
\end{lem}
\begin{proof}
    Since $K_X+B\equiv 0$, we have $K_{X'}+B'=\phi_{*}(K_X+B)\equiv 0$.
    Then the negativity lemma shows that $K_X+B\equiv\phi^{*}(K_{X'}+B')$.
    Thus, we have $e(X,B,E')=e(X',B',E')$ for every divisor $E'$ over $X'$ (\cite[Lemma 2.30]{Kollar-Mori}).
    Finally, $K_{X'}+B'$ is $\Q$-Cartier by assumption, we obtain the assertion.
\end{proof}

\begin{lem}\label{lem:dlt}
    Let $(X,B)$ be a dlt surface pair over an algebraically closed field.  
    Let $B^{=1}=\sum_{i}B^{=1}_{i}$ be the irreducible decomposition of $B^{=1}$ and $F\coloneqq \bigcup_{i\neq j} (B^{=1}_{i}\cap B^{=1}_{j})$.
    Let $E$ be a prime exceptional divisor over $X$.
    Then $e(E,X,B)=1$ if and only if the center of $E$ is contained in $F$.
\end{lem}
\begin{proof}
    Note that $(X,B)$ is log smooth at every point of $F$ since it is dlt.
    Therefore, the ``if'' direction follows from the well-known calculation of the discrepancy of log smooth pairs (\cite[Corollary 2.31 (3)]{Kollar-Mori}).

   Now, we prove the ``only if'' direction.
   Suppose that $e(E,X,B)=1$.
   By the definition of dltness (\cite[Definition 2.37]{Kollar-Mori}), the center of $E$ is contained in the log smooth locus of $(X,B)$.
   Then, using \cite[Corollary 2.31 (3)]{Kollar-Mori} again, we conclude that the center must be contained in $F$.
\end{proof}

\begin{lem}\label{lem:extraction}
Let $(X,B)$ be a dlt surface pair over an algebraically closed field.
Let $B^{=1}=\sum_{i}B^{=1}_{i}$ be the irreducible decomposition of $B^{=1}$ and $F\coloneqq \bigcup_{i\neq j} (B^{=1}_{i}\cap B^{=1}_{j})$.
Let $E$ be a prime exceptional divisor over $X$ such that $e(E,X,B)\geq 0$ and the center is not contained in $F$.
Then $e(E,X,B)<1$ and
there exists a proper birational morphism $f\colon Y\to X$ satisfying the following properties:
\begin{enumerate}
        \item[\textup{(1)}] $\Exc(f)=E$ and
        \item[\textup{(2)}] $K_Y+f_{*}^{-1}B+e(E,X,B)E=f^{*}(K_X+B)$.
    \end{enumerate}
We call $f$ the \textit{extraction} of $E$.
\end{lem}
\begin{proof}
We take a log resolution $g\colon W\to X$ of $(X,B)$ such that $E$ is a divisor on $W$.
Since the center of $E$ is not contained in $F$, we can assume that $g$ is an isomorphism over a neighborhood of $F$.
We can write the equation as follows:
\[
K_{W}+g_{*}^{-1}B+\sum_{i} a_{i}E_{i}=g^{*}(K_X+B)
\]
where the sum runs over all the $g$-exceptional divisors.
Since the center of $E_i$ is not contained in $F$, we have $a_{i}<1$ for every $i$ by Lemma \ref{lem:dlt}.
By changing the order of the exceptional divisors, we may assume that $E_1=E$ and $a_{1}=e(E, X, B)$.

Next, we run a $(K_W+g_{*}^{-1}B+a_{1}E_{1}+\sum_{i\geq 2} E_{i})$-MMP over $X$ to obtain a birational contraction $\phi\colon W\to Y$ and the minimal model $f\colon Y\to X$ over $X$.
Since $K_W+g_{*}^{-1}B+a_{1}E_{1}+\sum_{i\geq 2} E_{i}\equiv_{X} \sum_{i\geq 2}(1-a_i)E_i$, it follows that $\Exc(\phi)\subset \sum_{i\geq 2}E_i$.
Since $K_Y+f_{*}^{-1}B+a_{1}\phi_{*}E_{1}+\sum_{i\geq 2} \phi_{*}E_{i}\equiv_{X} \sum_{i\geq 2}(1-a_i)\phi_{*}E_i$ is nef over $X$, the negativity lemma shows that $\Exc(\phi)=\sum_{i\geq 2}E_i$.
Therefore, we have $\Exc(f)=E_1$ and
\[
K_Y+f_{*}^{-1}B+a_1E_1=f^{*}(K_X+B),
\]
as desired.
\end{proof}

\subsection{Log liftability}

In this subsection, we discuss basic properties of log liftability.

\begin{defn}[Log liftability]
Let $(X, B)$ be a projective surface pair over an algebraically closed field $k$ of positive characteristic.
Let $R$ be a Noetherian complete local ring with residue field $k$.
We say that $(X, B)$ is \textit{log liftable to $R$} if there exists a log resolution $f\colon Y \to X$ of $(X, B)$ such that $(Y,f_{*}^{-1}\Supp(B)+\Exc(f))$ lifts to $R$.
For the definition of liftability of a log smooth pair, we refer to \cite[Definition 2.6]{Kaw3}.
When $R$ is the ring $W(k)$ of Witt vectors, we simply say that $(X,B)$ is \textit{log liftable}.
\end{defn}

\begin{lem}\label{lem:push}
Let $(X, B)$ be a projective surface pair over an algebraically closed field $k$ of positive characteristic.
Let $\phi\colon X\to X'$ be a birational morphism to a normal projective surface $X'$ and $B'\coloneqq \phi_{*}B$.
Let $R$ be a Noetherian complete local ring with residue field $k$.
Then the followings hold.
\begin{enumerate}
    \item[\textup{(1)}] If $(X, B)$ is log liftable to $R$ and $\Exc(\phi)\subset \Supp(B)$, then $(X', B')$ is log liftable to $R$.
    \item[\textup{(2)}] If $(X', B')$ is log liftable to $R$ and $R$ is regular, then $(X, B)$ is log liftable to $R$.
\end{enumerate}
\end{lem}
\begin{proof}
First, we prove (1). Suppose that $(X,B)$ is log liftable to $R$. Then there exists a log resolution $f\colon Y\to X$ such that $(Y, f_{*}^{-1}\Supp(B)+\Exc(f))$ lifts to $R$. We have a commutative diagram as follows:
\[ 
\xymatrix{
 (Y, f_{*}^{-1}\Supp(B)+\Exc(f))\ar[d]_-{f}\ar[rd]^-{\phi \circ f} & \\
 (X,\Supp(B))\ar[r]_{\phi} &   (X', \Supp(B'))     . \\
}
\]
Clearly, we have 
\[
(\phi \circ f)^{-1}_{*}\Supp(B')=f_{*}^{-1}(\phi_{*}^{-1}\Supp(B'))\subset f_{*}^{-1}\Supp(B).
\]
Since $f_{*}^{-1}\Exc(\phi)\subset f_{*}^{-1}\Supp(B)$ by assumption, we also have 
\[
\Exc(\phi \circ f)=f_{*}^{-1}\Exc(\phi)+\Exc(f)\subset f_{*}^{-1}\Supp(B)+\Exc(f).
\]
In particular, we obtain
\[
(\phi \circ f)^{-1}_{*}\Supp(B')+\Exc(\phi \circ f)\subset f_{*}^{-1}\Supp(B)+\Exc(f).
\]
 Thus $\phi \circ f$ is a log resolution of $(X',B')$ and $(Y, (\phi \circ f)^{-1}_{*}\Supp(B)+\Exc(\phi \circ f))$ lifts to $R$, i.e., $(X',B')$ is log liftable to $R$.

Next, we prove (2).
Suppose that $(X', B')$ is log liftable to $R$ and $R$ is regular.
Then there exists a log resolution $f'\colon Y'\to X'$ of $(X',B')$ such that $(Y', (f')_{*}^{-1}\Supp(B')+\Exc(f'))$ lifts to $R$.
We take a resolution $g'\colon W'\to Y'$ of the indeterminacy of the rational map $\phi^{-1}\circ f'\colon Y'\dashrightarrow X$ and denote the obtained morphism by $h'\colon W'\to X$. We obtain a commutative diagram as follows:
 \[
\xymatrix{
 (W', (f'\circ g')_{*}^{-1}\Supp(B')+\Exc(f'\circ g'))\ar[d]_-{h'}\ar[r]^-{g'} & (Y', (f')_{*}^{-1}\Supp(B')+\Exc(f')) \ar[d]^-{f'}\ar@{.>}[ld] \\
 (X,\Supp(B))\ar[r]_{\phi} &   (X', \Supp(B'))     . \\
}
\]
Since $g'$ is a composition of blow-ups at smooth points and $R$ is regular, it follows from \cite[Lemma 2.8]{Kaw3} that $(W', (f'\circ g')_{*}^{-1}\Supp(B')+\Exc(f'\circ g'))$ is log smooth and lifts to $R$.
 Note that 
 \[
 \Exc(h')\subset \Exc(\phi\circ h')=\Exc(f'\circ g')
 \] 
 and 
 \begin{align*}
     (h')_{*}^{-1}\Supp(B)&\subset(\phi\circ h')_{*}^{-1}\Supp(B')+(h')_{*}^{-1}\Exc(\phi)\\
     &\subset (f'\circ g')_{*}^{-1}\Supp(B')+\Exc(f'\circ g').
 \end{align*}
 In particular, we have \[(h')_{*}^{-1}\Supp(B)+\Exc(h')\subset (f'\circ g')_{*}^{-1}\Supp(B')+\Exc(f'\circ g').\]
 Since $(W', (f'\circ g')_{*}^{-1}\Supp(B')+\Exc(f'\circ g'))$ is log smooth and lifts to $R$, it follows that $h'\colon W'\to X$ is a log resolution of $(X,B)$ and $(W', (h')_{*}^{-1}\Supp(B)+\Exc(h'))$ lifts to $R$. Therefore, $(X, B)$ is log liftable to $R$.

\end{proof}

\begin{thm}\label{thm:log lift criterion}
Let $X$ be a projective surface pair over an algebraically closed field $k$ of positive characteristic such that $B$ is reduced.
Let $R$ be a Noetherian complete local ring with residue field $k$.
Suppose that $H^2(X, T_{X}(-\log\,B))=0$ and $H^2(X, \sO_X)=0$.
Then $(X,B)$ is log liftable to $R$.
\end{thm}
\begin{proof}
 Let $f\colon Y\to X$ be a log resolution of $(X,B)$. 
By \cite[Remark 4.2]{Kaw3}, we have 
\[
H^2(T_Y(-\log\,f^{-1}_{*}B+\Exc(f)))\hookrightarrow H^2(X, T_X(-\log\,B))=0.
\]
We also have 
\[
H^2(Y, \sO_Y)\cong H^0(Y, \sO_Y(K_Y))\hookrightarrow H^0(X, \sO_X(K_X))\cong H^2(X, \sO_X)=0,
\]
where we use \[f_{*}\sO_Y(K_Y)\hookrightarrow (f_{*}\sO_Y(K_Y))^{**}=\sO_X(f_{*}K_Y)=\sO_X(K_X)\] for the second injective map.
Then it follows from \cite[Theorem 2.10]{Kaw3} that $(Y, f^{-1}_{*}\Supp(B)+\Exc(f))$ lifts to $R$.
Therefore, $(X,B)$ is log liftable to $R$.
\end{proof}

\begin{lem}\label{lem:log lift criterion 2}
Let $I\subset [0,1]\cap \Q$ be a finite set of rational numbers.
There exists a positive integer $p(I)\in\Z_{>0}$ depending only on $I$ with the following property:
For every projective surface pair $(X,B)$ over an algebraically closed field of characteristic $p>p(I)$ satisfying the following conditions:
\begin{enumerate}
\item[\textup{(1)}] The coefficients of $B$ belong to $I$,
\item[\textup{(2)}] there exists a positive integer $m(I)\in\Z_{>0}$ depending only on $I$ such that $m(I)B$ is Cartier, and
\item[\textup{(3)}] there exists a very ample divisor $H$ on $X$ such that there are only finitely many possibilities for $H^2$, $H\cdot K_{X}$, $H\cdot B$, $K_{X}\cdot B$, $B^2$, $\dim\,H^0(X, \sO_{X}(H))$, and $\chi(X, \sO_X)$,
\end{enumerate}
the pair $(X,B)$ is log liftable.
\end{lem}
\begin{proof}
    We first show the following claim.
\begin{cl}
There exists a flat family $(\mathcal{X},\mathcal{B})$ of pairs over a reduced quasi-projective scheme $T$ over $\Spec\,\Z$ such that every projective surface pair $(X,B)$ over an algebraically closed field of characteristic $p>5$ satisfying conditions (1)--(3) is a geometric fiber of $(\mathcal{X},\mathcal{B})$.
\end{cl}
\begin{clproof}
    By the Riemann-Roch theorem, the Hilbert polynomials of $X$ and $nB$ with respect to $H$ are given by:
    \[
    \frac{H^2}{2}t^2-\frac{H\cdot K_X}{2}t+\mathcal{X}(X, \sO_X)
    \] and 
    \[(nB\cdot H)t-\frac{1}{2}nB(nB+K_X)
    \] 
    respectively (see \cite[Remark 6.3]{Wit}).
    Therefore, the claim can be proved using essentially the same argument as \cite[Lemma 3.1]{CTW}.
\end{clproof}
    
Now, the claim and the proof of \cite[Proposition 3.2]{CTW} shows that there exists a positive integer $p(I)\in\Z_{>0}$ depending only on $I$ with the following property: For every projective surface pair $(X,B)$ over an algebraically closed field of characteristic $p>p(I)$ satisfying conditions (1)--(3), we can take a log resolution $f\colon Y\to X$ such that $(Y, f_{*}^{-1}\Supp(B)+\Exc(f))$ lifts to characteristic zero over a smooth base, as defined in \cite[Definition 2.15]{CTW}.

Since liftability to characteristic zero over a smooth base is equivalent to liftability to $W(k)$ by \cite[Proposition 2.5]{ABL}, we can conclude the assertion.
\end{proof}

\section{Proof of Theorems \ref{Intro, KVV} and \ref{Intro:Thm:log liftability of log CY}}

\subsection{Dlt log Calabi-Yau surface pairs}
There exists a positive real number $\epsilon\in\R_{>0}$ such that every klt log Calabi-Yau surface pair $(X,0)$ over every algebraically closed field is $\epsilon$-klt (\cite[Lemma 3.9]{Kaw3}).
In this subsection, we aim to generalize this result to the case of dlt log Calabi-Yau surface pairs (Proposition \ref{prop:epsilon-dlt}).

To begin, we recall the lemma known as the global ACC:

\begin{lem}\label{lem:global ACC}
Let $I\subset [0,1]\cap \Q$ be a DCC set of rational numbers. Then there exists a finite subset $J\subset I$ with the following property: For every projective surface pair $(X,B)$ over an algebraically closed field satisfying the following conditions:
\begin{enumerate}
\item[\textup{(1)}] $(X,B)$ is log Calabi-Yau and
\item[\textup{(2)}] the coefficients of $B$ belong to $I$,
\end{enumerate}
all the coefficients of $B$ belong to $J$.
\end{lem}
\begin{proof}
This is \cite[Lemma 3.8]{Kaw3}.
\end{proof}

\begin{prop}\label{prop:epsilon-dlt}
Let $I\subset [0,1]\cap \Q$ be a DCC set of rational numbers.
There exists a positive real number $\epsilon(I)\in\R_{>0}$ depending only on $I$ with the following property:
For every projective surface pair $(X, B)$ over an algebraically closed field satisfying the following conditions:
\begin{enumerate}
\item[\textup{(1)}] $(X,B)$ is dlt log Calabi-Yau and
\item[\textup{(2)}] the coefficients of $B$ belong to $I$,
\end{enumerate}
the pair $(X, B^{<1})$ is $\epsilon(I)$-klt.
\end{prop}
\begin{proof}
By Lemma \ref{lem:global ACC}, the coefficients of $B$ belong to a finite subset of $I$.
Therefore, it suffices to show that there exists a positive real number $\epsilon(I)\in\R_{>0}$ such that $e(E,X,B^{<1})< 1-\epsilon(I)$ for every exceptional divisor $E$ over $X$.

Suppose, by contradiction, that there exists a sequence $\{(X_l,B_l)\}_{l\in \Z_{>0}}$ of projective surface pairs over algebraically closed fields satisfying the conditions (1) and (2), and $\{e(E_{l}^{\max},X_l,B_{l}^{<1})\}_{l\in \Z_{>0}}$ converges to $1$, where $E_l^{\max}$ is a prime exceptional divisor over $X_l$ with the largest coefficient with respect to $(X_l,B_l^{<1})$ among all exceptional divisors.
By taking a subsequence if necessary, we may assume that $e(E_{l}^{\max},X_l,B^{<1}_l)>0$ for every $l$.

Let $B_l^{=1}=\sum_i B^{=1}_{l,i}$ be the irreducible decomposition of $B_l^{=1}$.
We set $F_l\coloneqq \bigcup_{i\neq j} (B^{=1}_{l,i}\cap B^{=1}_{l,j})$.
Suppose that the center of $E_l^{\max}$ is contained in $F_l$. Since $(X_l,B_l)$ is dlt, $X_l$ is smooth at the center of $E_l^{\max}$. Thus, we have $e(E^{\max}_l,X_l,B_{l}^{<1})=-1$, which contradicts the fact that $e(E_{l}^{\max},X_l,B^{<1}_l)>0$.
Therefore, the center of $E_l^{\max}$ is not contained in $F_l$.
Moreover, we have $e(E_{l}^{\max},X_l,B_l)\geq e(E_{l}^{\max},X_l,B^{<1}_l)>0$.
Thus we can use Lemma \ref{lem:extraction}, and we have $e(E_{l}^{\max}, X_l,B_l)<1$ and the extraction $g_{l}\colon Z_{l}\to X_{l}$ of $E^{\max}_l$.
Since we have
\[
K_{Z_l}+(g_{l})^{-1}_{*}B_l+e(E_{l}^{\max}, X_l,B_l)E_{l}^{\max}=(g_l)^{*}(K_{X_l}+B_l),
\]
the pair $(Z_{l}, (g_{l})^{-1}_{*}B_l+e(E_{l}^{\max}, X_l,B_l)E_{l}^{\max})$ is log Calabi-Yau. 
Since $\{e(E_{l}^{\max},X_l,B_{l}^{<1})\}_{l\in \Z_{>0}}$ converges to $1$ and \[e(E_{l}^{\max},X_l,B_{l}^{<1})\leq e(E_{l}^{\max},X_l,B_{l})<1,\]
we can assume, by taking a subsequence if necessary, that $\{e(E_{l}^{\max},X_l,B_{l})\}_{l\in \Z_{>0}}$ is a strictly increasing sequence.
This contradicts Lemma \ref{lem:global ACC}.
\end{proof}

\subsection{The klt case of Theorem \ref{Intro:Thm:log liftability of log CY}}
In this subsection, we prove Theorem \ref{Intro:Thm:log liftability of log CY} when $(X,B)$ is klt (Proposition \ref{prop:klt case}).
When $X$ is klt and $B=0$, Theorem \ref{Intro:Thm:log liftability of log CY} has been proven in \cite[Section 3]{Kaw3}.
This proof also works even when $B\neq 0$ if $(X,B)$ is klt.

\begin{lem}\label{lem:canonical CY case}
Let $X$ be a normal projective surface over an algebraically closed field of characteristic $p>19$ such that $X$ has only canonical singularities and $K_X\equiv 0$.
Then $X$ is log liftable.
\end{lem}
\begin{proof}
This is \cite[Propostion 3.3]{Kaw3}.
\end{proof}

\begin{defn}\label{MFS}
Let $(X, B)$ be a pair over an algebraically closed field.
Let $g \colon X \to Z$ be a projective surjective morphism to a normal variety $Z$ such that $g_{*}\sO_X=\sO_Z$.
We say that $g \colon X \to Z$ is a \textit{$(K_X+B)$-Mori fiber space} if
\begin{enumerate}
    \item[\textup{(1)}] $-(K_X+B)$ is $g$-ample,
    \item[\textup{(2)}] $\dim \, X>\dim \, Z$, and
    \item[\textup{(3)}] the relative Picard rank $\rho(X/Z)=1$.
\end{enumerate}
\end{defn}

\begin{prop}\label{prop:klt case}
Let $I\subset [0,1]\cap \Q$ be a DCC set of rational numbers.
There exists a positive integer $p(I)$ depending only on $I$ with the following property:
For every projective surface pair $(X, B)$ over an algebraically closed field of characteristic $p>p(I)$ satisfying the following conditions:
\begin{enumerate}
    \item[\textup{(1)}] $(X,B)$ is klt log Calabi-Yau, and
    \item[\textup{(2)}] the coefficients of $B$ belong to $I$,
\end{enumerate}
the pair $(X, B)$ is log liftable.
\end{prop}
\begin{proof}
By Lemma \ref{lem:global ACC}, we can assume that $I$ is a finite set. Furthermore, by Proposition \ref{prop:epsilon-dlt}, we can find $\epsilon(I)$ depending only on $I$, such that every projective surface pair $(X, B)$ over an algebraically closed field satisfying conditions (1) and (2) is $\epsilon(I)$-klt.
Then we can use \cite[Lemma 3.6]{Kaw3} to find a uniform bound $m(I)$ on the $\Q$-factorial index, depending only on $I$. We replace $I$ with $I\cup\{\frac{1}{m(I)},\ldots, \frac{m(I)-1}{m(I)}\}$.

\textbf{Step 1: The case where $B\neq 0$.}\,\,
In this step, we aim to find the desired positive integer $p(I)$ under the additional assumption that $B\neq 0$.

Let $(X, B)$ a projective surface pair over an algebraically closed field of characteristic $p$ satisfying the conditions (1) and (2).
Since $K_{X}\equiv -B$ is not pseudo-effective, we can run a $K_X$-MMP to obtain a birational contraction $\phi\colon X\to X'$ and a $K_{X'}$-Mori fiber space $X'\to Z$. By Lemma \ref{lem:log CY}, the pair $(X', B'\coloneqq \phi_{*}B)$ is $\epsilon(I)$-klt log Calabi-Yau. Note that $K_{X'}+B'$ is $\Q$-Cartier because $X'$ is $\Q$-factorial. Additionally, using Lemma \ref{lem:push} (2), we can deduce that if $(X', B')$ is log liftable, then so is $(X,B)$. Finally, we can conclude the existence of the desired positive integer $p(I)$ by \cite[Lemma 3.14]{Kaw3}.

\textbf{Step 2: The case where $B=0$.}\,\,
We take $p(I)$ as defined above and replace it with $\max\{p(I), 19\}$. 
In this step, we aim to prove that a klt log Calabi-Yau surface pair $(X,0)$ over an algebraically closed field of characteristic $p>p(I)$ is log liftable, thus completing the proof of the proposition.

We take $X=(X,0)$ as above.
If $X$ is canonical, then it is log liftable by Lemma \ref{lem:canonical CY case} since $p>19$. 
We assume that $X$ is not canonical.
By Lemma \ref{lem:extraction}, we can take an extraction $f\colon Y\to X$ of a divisor $E$ over $X$ with the maximum coefficient $e(E, X, 0)$. Note that $e(E, X, 0)>0$ since $X$ is not canonical. Since $K_Y+e(E,X,0)E\equiv f^{*}K_X$, the pair $(Y, e(E,X,0)E)$ is $\epsilon(I)$-klt log Calabi-Yau. It suffices to show that $(Y, e(E,X,0)E)$ is log liftable by Lemma \ref{lem:push} (1).
Considering that $m(I)K_X$ is Cartier, we can conclude that \[e(E, X, 0)\in\{\frac{1}{m(I)},\ldots, \frac{m(I)-1}{m(I)}\}\subset I.\] 
Thus, since $p>p(I)$, the pair $(X,B)$ is log liftable by the previous step.
\end{proof}

\subsection{The non-klt case of Theorem \ref{Intro:Thm:log liftability of log CY}}
In this subsection, we prove the non-klt case of Theorem \ref{Intro:Thm:log liftability of log CY} (Proposition \ref{prop:non-klt case}). Unlike the klt case, the non-klt case does not follow from the direct generalization of \cite[Section 3]{Kaw3}, and we require additional ingredients.

Proposition \ref{prop:non-klt case} is reduced to the log liftability of certain dlt Mori fiber spaces. When the base of the Mori fiber space is a point, we prove the log liftability by combining Lemma \ref{lem:log lift criterion 2} and Proposition \ref{prop:epsilon-dlt}. When the base of the Mori fiber space is a curve, we utilize Lemma \ref{lem:fiber cases}. In this lemma, we prove the vanishing of the second cohomology of the logarithmic tangent sheaf in order to apply Theorem \ref{thm:log lift criterion}. To prove the vanishing, we require the assumption that the coefficients of boundary divisors are greater than or equal to $\frac{1}{2}$.

\begin{lem}\label{lem:fiber cases}
Let $(X,B)$ be a projective surface pair such that the coefficients of $B$ belong to $[\frac{1}{2},1]\cap \Q$.
Let $g\colon X\to Z$ be a surjective morphism to a smooth projective curve $Z$ such that $g_{*}\sO_X=\sO_Z$.
Suppose that 
\begin{enumerate}
    \item[\textup{(1)}] $-(K_X+B)$ is $g$-nef,
    \item[\textup{(2)}] there exists an irreducible component of $B^{=1}$ that dominates $Z$, and
    \item[\textup{(3)}] $p>2$.
\end{enumerate}
Then $(X,B)$ is log liftable.
\end{lem}
\begin{proof}
Let $F$ be a general fiber of $g\colon X \to Z$. 
Since $\dim\,Z=1$ and $g_{*}\sO_X=\sO_{Z}$, the fiber $F$ is an integral curve (\cite[Corollary 7.3]{Bad}).
By assumptions (1) and (2), we have $K_X\cdot F\leq -B\cdot F<0$.
Therefore, we obtain $F\cong \PP_k^1$ and $K_X\cdot F=\deg_{F}(\sO_F(K_F))=-2$ (\cite[Chapter 7, Proposition 4.1]{Liu}).
Moreover, $F$ is nef, as $F^2 = 0$. 

By Theorem \ref{thm:log lift criterion}, it suffices to show that 
\[
H^2(X, T_{X}(-\log\,\Supp(B)))=H^2(X, \sO_X)=0.
\]
Since $F$ is nef and $K_X\cdot F=-2$, we have $H^2(X, \sO_X)\cong H^0(X, \sO_X(K_X))=0$.
We show that 
\[
H^2(X, T_{X}(-\log\,\Supp(B)))=H^0(X, (\Omega^{[1]}_{X}(\log\,\Supp(B))\otimes \sO_X(K_X))^{**})=0.
\]
To do so, it suffices to show that 
\[
g_{*}(\Omega^{[1]}_{X}(\log\,\Supp(B))\otimes \sO_X(K_X))^{**}=0.
\]
Since this sheaf is torsion-free, it is enough to prove that it has rank zero. This property is local on $Z$, so we can shrink $Z$ if necessary. In particular, we may assume that $Z$ is affine, $X$ is smooth, and all the components of $B$ dominate $Z$.

Suppose by contradiction that 
\begin{align*}
  0\neq g_{*}(\Omega^{[1]}_{X}(\log\,\Supp(B))\otimes \sO_X(K_X))^{**}=H^0(X, (\Omega^{[1]}_{X}(\log\,\Supp(B))\otimes \sO_X(K_X))^{**})
\end{align*}
Then we can take an injective morphism $s\colon\sO_{X}(-K_{X})\hookrightarrow \Omega^{[1]}_{X}(\log\,\Supp(B))$.

We prove the following claim.
\begin{cl}
    Let $B\coloneqq \sum_{i=0}^{n}b_iB_i$ be the irreducible decomposition.
    We have $B_i \cdot F\leq 2$ for every $i\geq 0$ and $(K_{X}+\lceil B\rceil)\cdot F\leq 1$.
\end{cl}
\begin{clproof}
We can assume that $B_0$ is an irreducible component of $B^{=1}$ that dominates $Z$ by assumption (2). 
By changing the order of the components, we can also assume that $B_i \cdot F \geq B_{i+1} \cdot F$ for $i \geq 1$.

Recalling that $-K_X \cdot F = 2$, $-(K_X + B) \cdot F \geq 0$, and $b_i \geq \frac{1}{2}$, we observe that one of the following must hold:
\begin{enumerate}
    \item $n=0$ and $B_0\cdot F\leq 2$,
    \item $n=1$, $B_0\cdot F=1$, and $B_1\cdot F\leq 2$, or
    \item $n=2$ and $B_0\cdot F=B_1\cdot F=B_2\cdot F=1$.
\end{enumerate}
Therefore, we can see that $B_i\cdot F\leq 2$ for every $i$, and $\lceil B\rceil\cdot F\leq3$, or equivalently, $(K_{X}+\lceil B\rceil)\cdot F\leq 1$ for every case, as required.   
\end{clproof}

By the claim, every irreducible component $B_i$ of $B$ is generically \'etale over $Z$, as otherwise we have a contradiction: $2 \geq B_i \cdot F \geq p >2$.
Thus we can shrink $Z$ so that $(X,B)$ is log smooth over $Z$, and we have the following diagram:
\begin{equation*}
\xymatrix{ & & \sO_{X}(-K_{X}) \ar@{.>}[ld]^{u} \ar[d]^{s} \ar[rd]^{t} &\\
                 0\ar[r] &\sO_X(g^{*}K_Z)\ar[r]   & \Omega_{X}(\log\,\Supp(B)) \ar[r]^-{\rho}  & \Omega_{X/Z}(\log\, \Supp(B)) \ar[r] & 0.}
\end{equation*}
The construction of the exact sequence is as follows. When $B = 0$, this is the usual relative differential sequence \cite[Proposition II.8.11]{Har}. When $B \neq 0$, we define the map $\rho$ as follows: $d(g^*z) \mapsto 0$ and $db/b \mapsto db/b$, where $z$ is a coordinate on $Z$ and $b$ is a local equation of $B$. Note that $g^*z$ and $b$ form a coordinate system on $X$ since $B$ is a simple normal crossing over $Z$.

By adjunction, we have \[\Omega_{X/Z}(\log\, \Supp(B))|_{F}=\sO_{F}(K_F+\lceil B\rceil|_F)=\sO_{F}(K_X+\lceil B\rceil).\] Using this, we obtain
\[
\deg_{F}(\sO_F(-K_{X}))=2>1\geq (K_{X}+\lceil B\rceil)\cdot F=\deg_{F}(\Omega_{X/Z}(\log\, \Supp(B))|_{F}).
\]
Thus, the map $t|_F$ must be the zero map. Since $F$ is a general fiber, this implies that $t$ is the zero map. Consequently, an injective homomorphism $u\colon \sO_{X}(-K_{X})\hookrightarrow \sO_{X}(g^{*}K_Z)$ is induced. Restricting $u$ to $F$ gives an injective map $u|_F$, where the injectivity follows from the generality of $F$.
Then we have \[\deg_F(\mathcal{O}_F(-K_X)) = 2 \leq \deg_F(\mathcal{O}_F(g^*K_Z)) = 0,\] a contradiction.
\end{proof}

\begin{prop}\label{prop:non-klt case}
Let $I \subset [\frac{1}{2}, 1] \cap \mathbb{Q}$ be a DCC set of rational numbers. There exists a positive integer $p(I) \in \mathbb{Z}_{>0}$ depending only on $I$ with the following property:
For every projective surface pair $(X, B)$ over an algebraically closed field of characteristic $p>p(I)$ satisfying the following conditions:
\begin{enumerate}
    \item[\textup{(1)}] $(X,B)$ is log Calabi-Yau, 
    \item[\textup{(2)}] the coefficients of $B$ belong to $I$, and
    \item[\textup{(3)}] $(X,B)$ is not klt
\end{enumerate}
the pair $(X, B)$ is log liftable.
\end{prop}
\begin{proof}
We may assume that $I$ is a finite set by Lemma \ref{lem:global ACC}. 
Let $(X,B)$ be a projective surface pair over an algebraically closed field of characteristic $p$ satisfying the conditions (1)--(3).
We aim to find the positive integer $p(I)$ as in the proposition.

Assuming  $1\in I$, we can apply Lemma \ref{lem:push} (1) to replace the pair $(X, B)$ with its dlt blow-up (see \cite[Definition 4.3]{Kaw3}). Since $(X,B)$ is not klt, we have $B^{=1} \neq 0$. By Proposition \ref{prop:epsilon-dlt}, there exists a positive real number $\epsilon(I) \in \mathbb{R}_{>0}$ that depends only on $I$, such that $(X, B^{<1})$ is $\epsilon(I)$-klt.

Since $K_X+B^{<1}\equiv -B^{=1}$ is not pseudo-effective, we run a $(K_X+B^{<1})$-MMP to obtain a birational contraction $\phi\colon (X, B^{<1}) \to (X', (B')^{<1})$ and a $(K_{X'}+(B')^{<1})$-Mori fiber space $g\colon X' \to Z$, where $B' \coloneqq \phi_{*}B$. 
Then $(X', (B')^{<1})$ is $\epsilon(I)$-klt, and Lemma \ref{lem:log CY} shows that $(X', B')$ is log Calabi-Yau. 
By Lemma \ref{lem:push} (2), if $(X',B')$ is log liftable, then so is $(X,B)$. Therefore, we can replace the pair $(X,B)$ with $(X',B')$.

First, we treat the case where $\dim\,Z = 1$. Since $K_X+B\equiv 0$ and $-(K_X+B^{<1})$ is $g$-ample, there exists an irreducible component of $B^{=1}$ that dominates $Z$. If $p>2$, then $(X,B)$ is log liftable by Lemma \ref{lem:fiber cases}.
Therefore, it suffices to choose $p(I)\geq 2$.

Next, we consider the case where $\dim\,Z = 0$. 
In this case, $(X,B^{<1})$ is an $\epsilon(I)$-klt log del Pezzo surface.
By Lemma \ref{lem:log lift criterion 2}, to find the desired positive integer $p(I)$, it suffices to confirm the following conditions:
\begin{enumerate}
    \item[(a)] There exists a positive integer $m(I) \in \mathbb{Z}_{>0}$ that depends only on $I$, such that $m(I)B$ is Cartier.
    \item[(b)] There exists a very ample divisor $H$ on $X$ such that there are only finitely many possibilities for $H^2$, $H \cdot K_X$, $H \cdot B$, $K_X \cdot B$, $B^2$, $\dim\,H^0(X, \mathcal{O}_X(H))$, and $\chi(X, \mathcal{O}_X)$.
\end{enumerate}

First, we check (a).
We have a uniform bound on the $\mathbb{Q}$-factorial index, which depends only on $I$, by \cite[Proposition 6.1 (d)]{Wit}. Therefore, recalling the fact that $I$ is a finite set, the condition (a) is satisfied.

Next, we check (b). Suppose that $p>5$. Then there exists a very ample divisor $H$ such that there are only finitely many possibilities for $H^2$, $H \cdot K_X$, and $H^i(X, \mathcal{O}_X(H)) = 0$ for $i > 0$ by \cite[Corollary 1.4]{Wit}. 
Moreover, using \cite[Lemma 6.2]{Wit}, we can conclude that there are only finitely many possibilities for $K_X^2$. Since $B \equiv -K_X$, there are also only finitely many possibilities for $H \cdot B$, $K_X \cdot B$, and $B^2$.

Let $f\colon W\to X$ be a resolution. Then $W$ is a smooth rational surface.
Since $X$ has only rational singularities, we have $H^i(X, \sO_X(D))\cong H^i(W, \sO_W(f^{*}D))$ for $i\geq 0$ and every Cartier divisor $D$ on $X$. 
In particular, we have $\mathcal{X}(X, \sO_X)=\mathcal{X}(W, \sO_W)=1$.
By the Riemann-Roch theorem, we have 
\begin{align*}
&\dim\,H^0(X, \sO_X(H))=\mathcal{X}(X, \sO_{X}(H))
                 =\mathcal{X}(W, \sO_{W}(f^{*}H))\\
                 =&\frac{(f^{*}H)^2}{2}+\frac{f^{*}H\cdot (-K_{W})}{2}+1
                 =\frac{(H)^2}{2}+\frac{H\cdot (-K_{X})}{2}+1,
\end{align*}
where we use the fact that $H^i(X, \sO_X(H))=0$ for $i>0$ for the first equality.
Thus there are only finitely many possibilities for $\dim\,H^0(X, \sO_X(H))$, and we have confirmed condition (b).

Therefore, we can find the desired positive integer $p(I)$.
\end{proof}

\subsection{Proof of main theorems}

\begin{proof}[Proof of Theorem \ref{Intro:Thm:log liftability of log CY}]
The assertion follows from Propositions \ref{prop:klt case} and \ref{prop:non-klt case}.
\end{proof}

\begin{proof}[Proof of Theorem \ref{Intro, KVV}]
By Theorem \ref{Intro:Thm:log liftability of log CY}, there exists a positive integer $p_0$ with the following property:
For every log Calabi-Yau surface pair $(X,B)$ over an algebraically closed field of characteristic $p>p_0$ such that $B$ has standard coefficients, the pair $(X,B)$ is log liftable.
We prove that this $p_0$ is the desired positive integer.

We take a log Calabi-Yau surface pair $(X,B)$ as above.
Let $f\colon Y\to X$ be a log resolution of $(X, B)$.
We set $A\coloneqq D-(K_X+\Delta)$.
By Serre duality for Cohen-Macaulay sheaves (\cite[Theorem 5.71]{Kollar-Mori}), we have 
\begin{align*}
H^i(X, \sO_X(D))\cong H^{2-i}(X, \sO_X(K_X-D))&=H^{2-i}(X, \sO_X(\lfloor K_X-D+\Delta\rfloor))\\
&=H^{2-i}(X, \sO_X(\lfloor-A\rfloor))\\
&=H^{2-i}(X, \sO_X(-A)),
\end{align*}
and it suffices to show that $H^i(X, \sO_X(-A))=0$ for $i<2$.
When $i=0$, the vanishing follows from the bigness of $A$.
We prove that $H^1(X, \sO_X(-A))=0$.
By the spectral sequence and the projection formula for $\Q$-divisors, we have an injective map
\[
H^1(X, \sO_X(-A))=H^1(X, f_{*}\sO_Y(- f^{*}A ))\hookrightarrow H^1(Y, \sO_Y(- f^{*}A )).
\]
By Theorem \ref{Intro:Thm:log liftability of log CY}, the pair $(Y, f_{*}^{-1}\Supp(B)+\Exc(f))$ lifts to $W(k)$.
Since $\Supp(\{f^{*}A\})\subset f_{*}^{-1}\Supp(B)+\Exc(f)$, it follows from \cite[Theorem 2.11]{Kaw3} that $H^1(Y, \sO_Y(- f^{*}A ))=0$, as desired.
\end{proof}

%%%%%%%%%%%%%%%%%%%%%%%%%%%%%%%%%%%%%%%%%%%%%%%%%%%%%%%%%%%%%%%%%%%%%%%%%%%%%%%%%%%%%%%%%%%%%%%%%%%%%%%%%%%%%%%%%%%%%%%%%%%%%%%%%%%%%%%%%%%%%%%%%%

\section*{Acknowledgements}
The author wishes to express his gratitude to Professor Shunsuke Takagi, Shou Yoshikawa, and Fabio Bernasconi for their helpful comments.
He is also grateful to the anonymous referee for pointing out some mistakes.
This work was supported by JSPS KAKENHI Grant Numbers JP19J21085 and JP22J00272.

\newcommand{\etalchar}[1]{$^{#1}$}

% \bibliography{hoge.bib}
% \bibliographystyle{alpha}

\end{document}